\documentclass[12pt, a4paper,reqno]{amsart}
\usepackage[margin=2.8cm, marginpar=2cm]{geometry}
\usepackage{verbatim}
\usepackage{graphicx}
\usepackage{color}
\usepackage{empheq}
\usepackage[font={small,it}]{caption}
\usepackage[all]{xy}
\usepackage{tikz-cd}
\usepackage{tikz}
\usetikzlibrary{calc}
\usepackage{graphicx}
\usepackage{array,booktabs,calc}
\usepackage{amsmath, amsthm, amssymb, amsfonts, amscd}
\usepackage[english]{babel}
\usepackage{url}
\usepackage{hyperref}
\usepackage{mathtools}
\usepackage{pinlabel}
\usepackage{scalerel}
\usepackage{cellspace}
\newcommand\restr[2]{{
  \left.\kern-\nulldelimiterspace 
  #1 
  \vphantom{\big|} 
  \right|_{#2}
  }}

\numberwithin{equation}{section}
\theoremstyle{remark}
\newtheorem{rmk}{Remark}
\theoremstyle{plain}
\newtheorem{thm}{Theorem}
\newtheorem{lem}[thm]{Lemma}
\newtheorem{prop}[thm]{Proposition}
\newtheorem{defi}[thm]{Definition}
\newtheorem{cor}[thm]{Corollary}

\theoremstyle{definition}
\newtheorem{exa}[thm]{Example}
\numberwithin{thm}{section}

\newcommand{\N}{\mathbb{N}}

\newcommand{\Z}{\mathbb{Z}}

\newcommand{\F}{\mathbb{F}}

\usepackage[foot]{amsaddr}
\makeatletter
\renewcommand{\email}[2][]{%
  \ifx\emails\@empty\relax\else{\g@addto@macro\emails{,\space}}\fi%
  \@ifnotempty{#1}{\g@addto@macro\emails{\textrm{(#1)}\space}}%
  \g@addto@macro\emails{#2}%
}
\makeatother

\title[]{Filtered matchings and simplicial complexes}

\author{Daniele Celoria}
\address[1,2]{University of Oxford, Mathematical Institute, Oxford, OX2 6GG, UK}
\email[1]{celoria@maths.ox.ac.uk}
\author{Naya Yerolemou}
\address[2]{The Alan Turing Institute, London, NW1 2DB, UK}
\email[2]{yerolemou@maths.ox.ac.uk}

\begin{document}

\begin{abstract}
To any finite simplicial complex $X$, we associate a natural filtration starting from Chari and Joswig's discrete Morse complex and abutting to the matching complex of $X$. This construction leads to the definition of several homology theories, which we compute in a number of examples. We also completely determine the graded object associated to this filtration in terms of the homology of simpler complexes. This last result provides some connections to the number of vertex-disjoint cycles of a graph.   
\end{abstract}
\maketitle

\section{Introduction}

In recent years there has been a growing interest in the combinatorial and topological properties of certain complexes arising from graph theory that can be associated to a simplicial complex. Namely, to a simplicial complex $X$ we can associate its matching complex; this is the simplicial complex $\text{M}(X)$ built from matchings on the face poset of $X$. 

For a graph $G$, the matching complex can be directly defined as the simplicial complex of matchings of $G$; this construction has been extensively studied for some special families of graphs, namely the complete and bipartite graph (see e.g.~\cite{shareshian2007torsion}, \cite{jonsson2007simplicial} and their references), due to the numerous connections between these complexes and several other interesting topics (see \cite[Ch.~1]{jonsson2007simplicial} for an exhaustive list).

A noteworthy subcomplex of the matching complex was introduced by Chari and Joswig in \cite{chari2005complexes}; this is the simplicial complex built from suitable equivalence classes of Forman's discrete Morse functions \cite{forman1998morse}.

The relation between these two complexes is not well understood. The aim of this paper is to introduce a natural filtration on the matching complex of a given finite simplicial complex, which will allow us to relate the matching complex to the discrete Morse complex.  Furthermore, we will define three homology theories related to the filtration, one of which is only valid when the simplicial complex is $1$-dimensional. 

While a full computation of these becomes rapidly unwieldy even for relatively simple complexes, it is possible to greatly reduce the complexity for one of these homology theories; more precisely, we prove in Theorem \ref{thm:hH} (Section~\ref{sec:filteredhomology}) that the homology of the associated graded object of the filtration can be decomposed as the direct sum of contributions from the discrete Morse complex and collections of oriented cycles in the face poset. 

We also provide several examples and computations of these homology groups.

\section{The matching and discrete Morse complexes}

In this section we recall the terminology and ideas we will use throughout this paper. We point out that in what follows we will use the terms ``cycle'' and ``loop'' interchangeably, specifying each time whether we consider them to be oriented or not if it is not clear from the context.

Given a finite simplicial complex $X$, one can associate to it its \emph{face poset} $\mathcal{F}(X)$. This is the oriented graph representing the poset structure of simplices in $X$. More precisely, vertices of $\mathcal{F}(X)$ are the simplices of $X$, and there is an oriented edge from $\sigma$ to $\tau$ if and only if $\tau$ is a codimension one face of $\sigma$.

A \emph{matching} on a graph $G$ is a subset $m \subseteq E(G)$ consisting of edges not sharing any vertices. Such a matching is called \emph{perfect} if all vertices in the graph are matched by some edge, that is, if the subgraph consisting of the matched edges is spanning.

The \emph{matching complex} $\text{M}(X)$ is the simplicial complex spanned by matchings in $\mathcal{F}(X)$. That is, the $k$ simplices are given by matchings with $k+1$ edges, and face maps are given by deleting one edge.

Given a matching on $\mathcal{F}(X)$, we can invert the orientation of all the edges that comprise the matching; if this operation does not create any oriented loops in the face poset, we say that the matching is \emph{acyclic}. A \emph{discrete Morse matching} on $X$ is an acyclic matching on $\mathcal{F}(X)$, and the \emph{discrete Morse complex} $\mathcal{M}(X)$ of $X$ is the simplicial complex spanned by discrete Morse matchings on $\mathcal{F}(X)$. This complex was first defined in \cite{chari2005complexes}; the terminology is derived from Forman's discrete Morse functions \cite{forman1998morse}, which are a simplicial analogue of the ``classical'' Morse functions on smooth manifolds. A discrete Morse matching can be thought of as being the gradient vector field of a discrete Morse function. With this point of view, (not necessarily acyclic) matchings are to be regarded as a discrete analogue of gradient vector fields of smooth functions.

There are several reasons that make the discrete Morse complex an interesting object to study, besides being a discrete analogue of its smooth counterpart; for instance it is a complete invariant of simplicial complexes \cite[Thm.~A]{capitelli2017simplicial}, meaning that two simplicial complexes are isomorphic if and only if their discrete Morse complexes are isomorphic. Furthermore, with the exception of the cycles $C_n$ and the boundary of the $n$-simplex, the automorphism group of $\mathcal{M}(X)$ is isomorphic to the automorphism group of $X$ by \cite[Thm.~1]{lin2019automorphism}.

We also point out this interesting result from \cite{chari2005complexes} (see also \cite{kozlov1999complexes}), which will be extensively used in the next sections. 
\begin{prop}{\cite[Prop.~3.1]{chari2005complexes}}\label{prop:chari}
If $G$ is a graph, there is a bijection between the set of discrete Morse matchings on $\mathcal{F}(G)$ and rooted spanning forests in $G$. 
\end{prop}

There is a bit of a notational conflict in the literature, in the case where $X$ is a graph. It is immediate to realise that if $G$ is a graph, then $\mathcal{F}(G)$ is isomorphic to the first barycentric subdivision $B(G)$ of $G$, with the orientation going from barycenters of the edges towards the vertices of $G$, as shown in Figure~\ref{fig:hassediagramgraph}. 
\begin{figure}[ht]
\centering
\includegraphics[width=6cm]{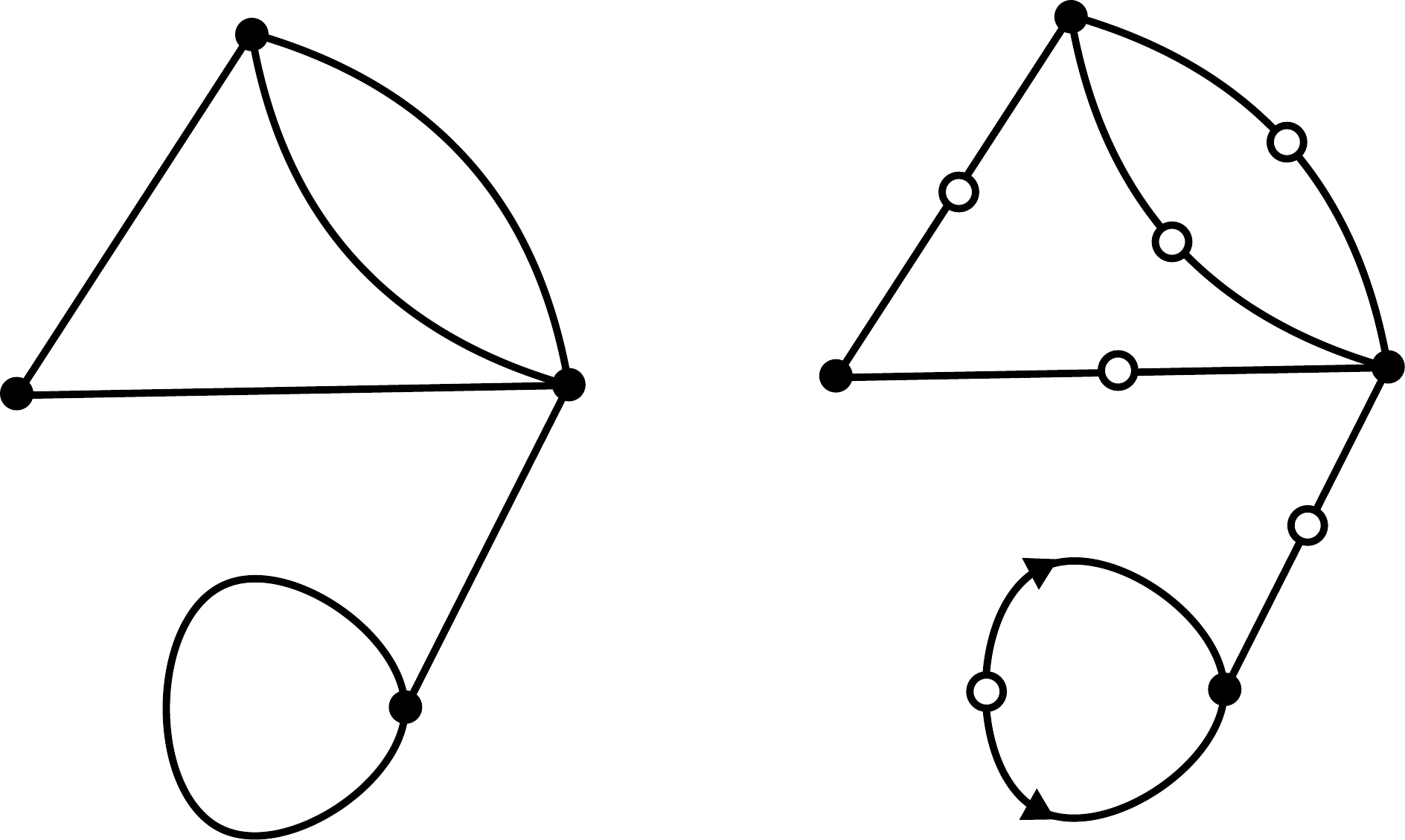}
\caption{On the left a graph, and on the right its face poset. White vertices represent the $1$-cells of $G$.}
\label{fig:hassediagramgraph}
\end{figure}
So in the case of graphs, we will instead denote by $\overline{\text{M}}(G)$ the simplicial complex spanned by disjoint edges of $G$. In other words, if $G$ is a $1$-dimensional simplicial complex (a graph), then $\text{M}(G) = \overline{\text{M}}(B(G))$. The complexes  $\overline{\text{M}}(G)$ have been extensively studied, especially for specific families of graphs --such as $K_n$, the complete graphs on $n$ vertices (see \cite{jonsson2007simplicial} for an extensive monograph on the subject). 

Analogously, given an oriented graph $G$, we define $\overline{\mathcal{M}}(G)$ as the simplicial complex with vertices given by $E(G)$, and simplices given by  collections of disjoint edges, such that changing the orientation of each of these edges does not create oriented loops. As before, for graphs we have $\mathcal{M} (G) = \overline{\mathcal{M}}(B(G))$. Observe that $\overline{\mathcal{M}}(G)$ is not well-defined for general graphs; for example, whenever $G$ contains an oriented loop, $\overline{\mathcal{M}}(G)$ cannot be a simplicial complex. We will, however, reserve this notation only for certain subgraphs of face posets, and in this case no issue arises.

If $F$ is a forest, it follows immediately that $\overline{\text{M}}(F) = \overline{\mathcal{M}}(F)$, since all matchings are acyclic. Moreover, a result by Marietti and Testa \cite[Thm.~4.13]{marietti2008uniform} guarantees that, for a forest, these complexes are either contractible or homotopic to a wedge of spheres (of possibly different dimensions).

For a simplex $\sigma \subseteq \text{M}(X)$, we will denote by $m_\sigma$ the matching on $\mathcal{F}(X)$ associated to $\sigma$.
Given a matching $m$ on $\mathcal{F}(X)$ inducing an oriented loop $\gamma$ in $\mathcal{F}(X)$, we say that $m$ \emph{supports} $\gamma$; an example is shown in Figure \ref{fig:supportingloop}. Note that if $m$ supports the loop $\gamma$, then every other edge of the loop $\gamma$
belongs to the matching $m$. However, the converse is not true: an unoriented loop with half of its edges alternatively belonging to a matching need not to be oriented (see the right part of Figure \ref{fig:supportingloop}).
\begin{lem}\label{lem:orientedloops}
Let $m$ be a matching on $\mathcal{F}(X)$. Then any oriented loop $\gamma$ induced on $\mathcal{F}(X)$ by $m$ has the following properties: 
\begin{enumerate}
\item The dimensions of the simplices corresponding to the vertices of $\gamma$ are two consecutive integers.
\item The loop $\gamma$  has even length.
\end{enumerate}
\end{lem}
\begin{proof}
Suppose, for a contradiction, that a loop $\gamma$ contains vertices corresponding to simplices in at least $3$ dimensions. Then, $\gamma$ necessarily contains a path of length $2$, starting at a vertex of dimension $k$ and ending at a vertex of dimension $k+2$. This implies that both edges in this path are contained in $m$, as they are both oriented upwards in dimension. This contradicts the fact that $\gamma$ is supported by a matching, as the middle vertex of dimension $k+1$ would necessarily be incident to two edges in $m$.

Hence, $\gamma$ is a loop in a bipartite graph (the subgraph of $\mathcal{F}(X)$ spanned by all vertices incident to edges in $m$) whose vertices are partitioned into two sets, those of dimension $l$ and those of dimension $l+1$, for some integer $l$. So, $\gamma$ has even length.
\end{proof}

\begin{figure}[ht]
\centering
\includegraphics[width=10cm]{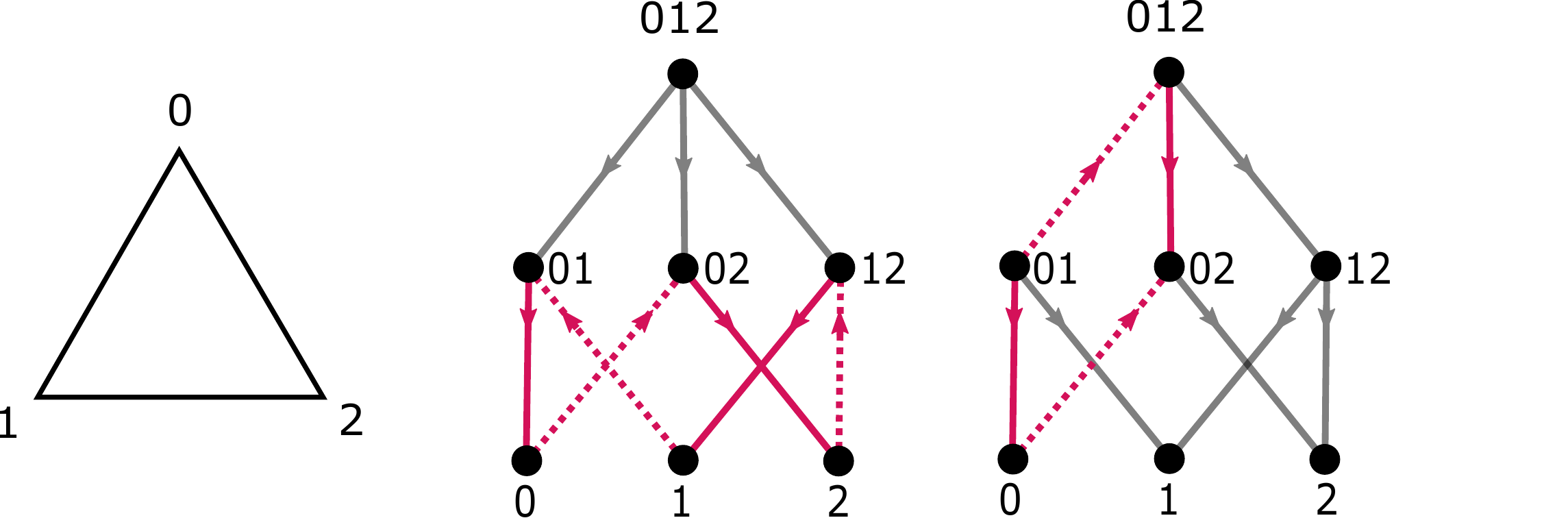}
\caption{On the left the $2$-simplex; the two graphs are a representation for its face poset. The matching given by the dotted edges in the central part supports the oriented red loop. Note that inverting the role of dotted and solid red edges produces another matching supporting the same loop, but with the opposite orientation. On the right, an example of a loop that is not supported by the matching, despite having half of its edges contained in it.}
\label{fig:supportingloop}
\end{figure}

For a simplicial complex $X$, we denote by $\mathcal{C}(X)$ the collection of unions of oriented cycles in $\mathcal{F}(X)$, such that each element $C \in \mathcal{C}(X)$ is supported by a matching in $\mathcal{F}(X)$. The empty set is included in $\mathcal{C}(X)$ by definition. Note that these cycles are not necessarily disjoint, as shown in Figure \ref{fig:nondisjointloops}. The number of oriented cycles in $C$ will be denoted by $|C|$, and the number of connected components in $C$ will be denoted by $\mu(C)$. We will also reserve the notation $\mu_1(C)$ for the number of cycles in $C$ not sharing  edges with any other one.

We say that $C \in \mathcal{C}(X)$ is \emph{maximal} if the subgraph of $\mathcal{F}(X)$ given by the loops in $C$ is spanning. Let us denote by $\mathcal{C}_\circ (X)$ the subset of  $\mathcal{C}(X)$ consisting of  non-maximal collections of cycles, and by $\mathcal{C}_M(X)$ the maximal ones. For reasons that will become apparent in the next section (\emph{cf}. Theorem~\ref{thm:hH}), we do not include the empty set of cycles in $\mathcal{C}_\circ(X)$, so $\mathcal{C}(X) = \{\emptyset\} \sqcup  \mathcal{C}_\circ(X) \sqcup \mathcal{C}_M(X)$.
\begin{figure}[ht]
\centering
\includegraphics[width=10cm]{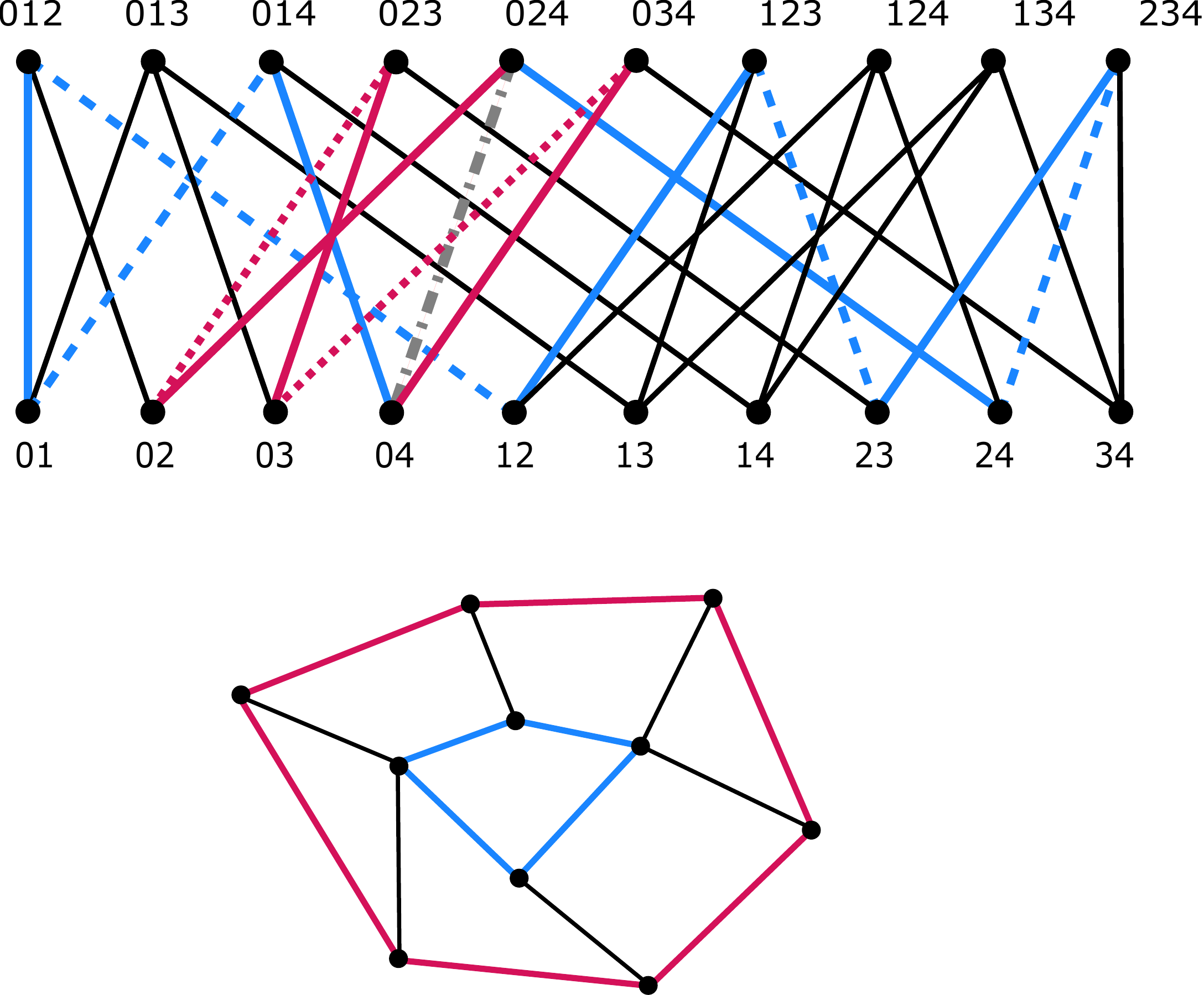}
\caption{At the top, an example where $C \in \mathcal{C}(X)$ consists of two oriented loops supported by a matching. One loop is blue, the other pink, and the common edge is in gray. Dotted and dashed edges indicate those belonging to the matching. The underlying graph is part of the face poset of the $4$-simplex. There is only one connected component, hence $\mu(C)=1$ and $\mu_1(C) = 0$. At the bottom, an example of a $2$-factor of a graph.}
\label{fig:nondisjointloops}
\end{figure}

Analogously, if $G$ is a graph, let us define $vd\mathcal{C}(G)$ as the collection of (unoriented) vertex-disjoint cycles in $G$ (including the empty set). A spanning element in $vd\mathcal{C}(G)$ is better known as a \emph{$2$-factor} of $G$. Finally, a \emph{pseudo-tree} is a graph with the homotopy type of $S^1$.
~\\

We can now define the map that will provide us with a filtration on $\text{M}(X)$:
\begin{defi}
Given a matching $m$ on $X$, define the value of the function $J$ on $m$ as the number of oriented cycles induced by $m$ on $\mathcal{F}(X)$. For an integer $k$, denote by $\text{M}_k (X)$ the set $J^{-1}([0, k])$, with the convention that $M_k(X) = \emptyset$ if $k < 0$.
\end{defi}

The next result ensures that the map $J$ does in fact provide a well-defined filtration:
\begin{prop}
The subsets $M_k (X)$ are simplicial subcomplexes of $M (X)$.
\end{prop}
\begin{proof}
It suffices to show that, given a matching $m$ on $\mathcal{F}(X)$ inducing $k$ oriented cycles in $\mathcal{F}(X)$, any face of the simplex $\sigma$ spanned by the edges of $m$ supports at most $k$ oriented cycles. That is, removing an edge from $m$ cannot increase the number of cycles supported by the reduced matching. 

Suppose, without loss of generality, that removing an edge $e$ from $m$ introduces one new cycle $\gamma$ supported by $m\setminus\{e\}$. Removing $e$ switches its orientation to downwards in the face poset. By Lemma \ref{lem:orientedloops}, the edges of $\gamma$ alternately belong to $m\setminus\{e\}$; this implies that the two vertices incident to $e$ must be incident to other edges in $m$, but this contradicts the fact that $m$ was a matching to begin with. 
\end{proof}

These subcomplexes are nested, meaning that for all integers $k$, $\text{M}_k (X) \subseteq \text{M}_{k+1} (X)$. Moreover, they interpolate between the discrete Morse and matching complexes on a given $X$; more precisely, $$\text{M}_0 (X) = \mathcal{M}(X)\ \text{and}\  \text{M}_{k \gg 0} (X) = \text{M} (X).$$
We denote by $\eta(X)$ the smallest integer $k$ such that $\text{M}_k (X) = \text{M} (X)$.
~\\

We can now give some sample computations related to this filtration for certain families of graphs and simplicial complexes.

\begin{exa}[\textbf{The circle}]\label{ex:circles}
We can explicitly determine the filtration and the homotopy type of the associated simplicial complexes $\text{M}_k$ for the graph $C_n$, the cycle graph on $n$ vertices. As unoriented graphs, $\mathcal{F}(C_n) \cong C_{2n}$, which can be seen by taking the first barycentric subdivision of $C_n$. Also, $\eta(C_n) = 1$, because clearly no matching can support more than one oriented cycle in $C_{2n}$.

The homotopy type of both the Morse and matching complexes for $C_n$ have been computed by Kozlov~\cite[Prop 5.2]{kozlov1999complexes}:
\begin{equation*}
\mathcal{M}(C_n) \simeq
\begin{cases}
S^{2k-1} \vee S^{2k-1} \vee S^{3k-2}  \vee S^{3k-2} &  \text{ if } n = 3k\\
S^{2k} \vee S^{3k-1} \vee S^{3k-1} & \text{ if } n = 3k+1\\
S^{2k} \vee S^{3k} \vee S^{3k} & \text{ if } n = 3k+2
\end{cases}
\end{equation*}

\begin{equation*}
\text{M}_1(C_n) = \text{M}(C_n) \simeq 
\begin{cases}
S^{2k-1} \vee S^{2k-1} &  \text{ if } n = 3k\\
S^{2k}  & \text{ if } n = 3k+1\\
S^{2k} & \text{ if } n = 3k+2
\end{cases}
\end{equation*} 

The difference between these two filtration levels is given by the two $(n-1)$-dimensional simplices shown in the right part of Figure \ref{fig:circles}.
\begin{figure}[ht]
\centering
\includegraphics[width=11cm]{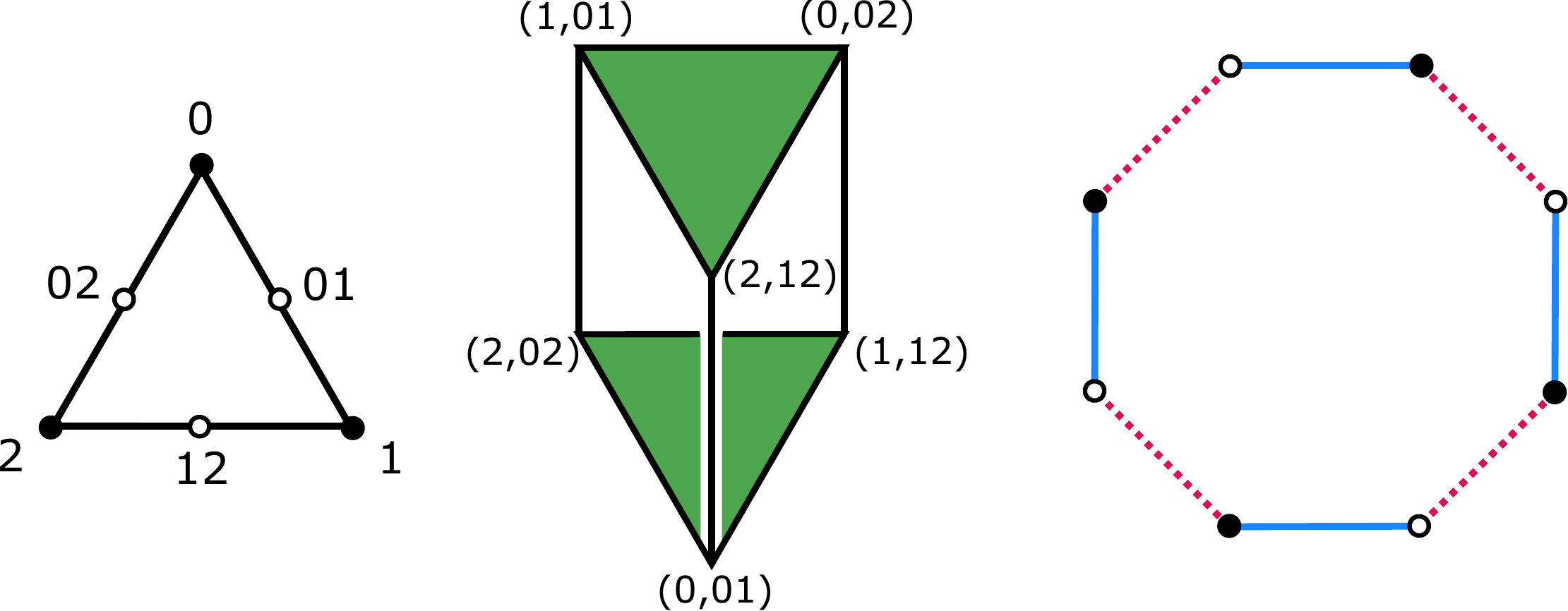}
\caption{On the left, the (barycentric subdivision of the) cycle $C_3$; in the centre, the complex $\text{M}_0(C_3)$ is given by the black $1$-skeleton; including the two green $2$-simplices gives $\text{M}_1(C_3)$.
On the right side, the only two (top-dimensional) matchings producing the simplices that appear in $\text{M}_1(C_4)$; one is pink and dotted, the other blue. Here, black and white dots represent the vertices in $V(C_4)$ and the edges' barycenters, respectively.}
\label{fig:circles}
\end{figure}
The net effect of adding these two top-dimensional simplices is to cap off two spheres generating the homology of the Morse complex.
\end{exa}

The following computations were provided using a Sage \cite{sagemath} program developed by the authors and available upon request. The notation $\Z_{(p)}^q$ denotes $q$ generators in homological degree $p$. 

\begin{exa}[\textbf{The simplices}]\label{ex:simplices}
The face poset of the $n$-simplex is the $1$-skeleton of the $(n+1)$-cube minus one vertex\footnote{We do not include the empty simplex in the diagram.}, with edges oriented away from the vertex corresponding to the whole simplex (see Figure \ref{fig:hassesimplex}).
\begin{figure}[ht]
\centering
\includegraphics[width=9cm]{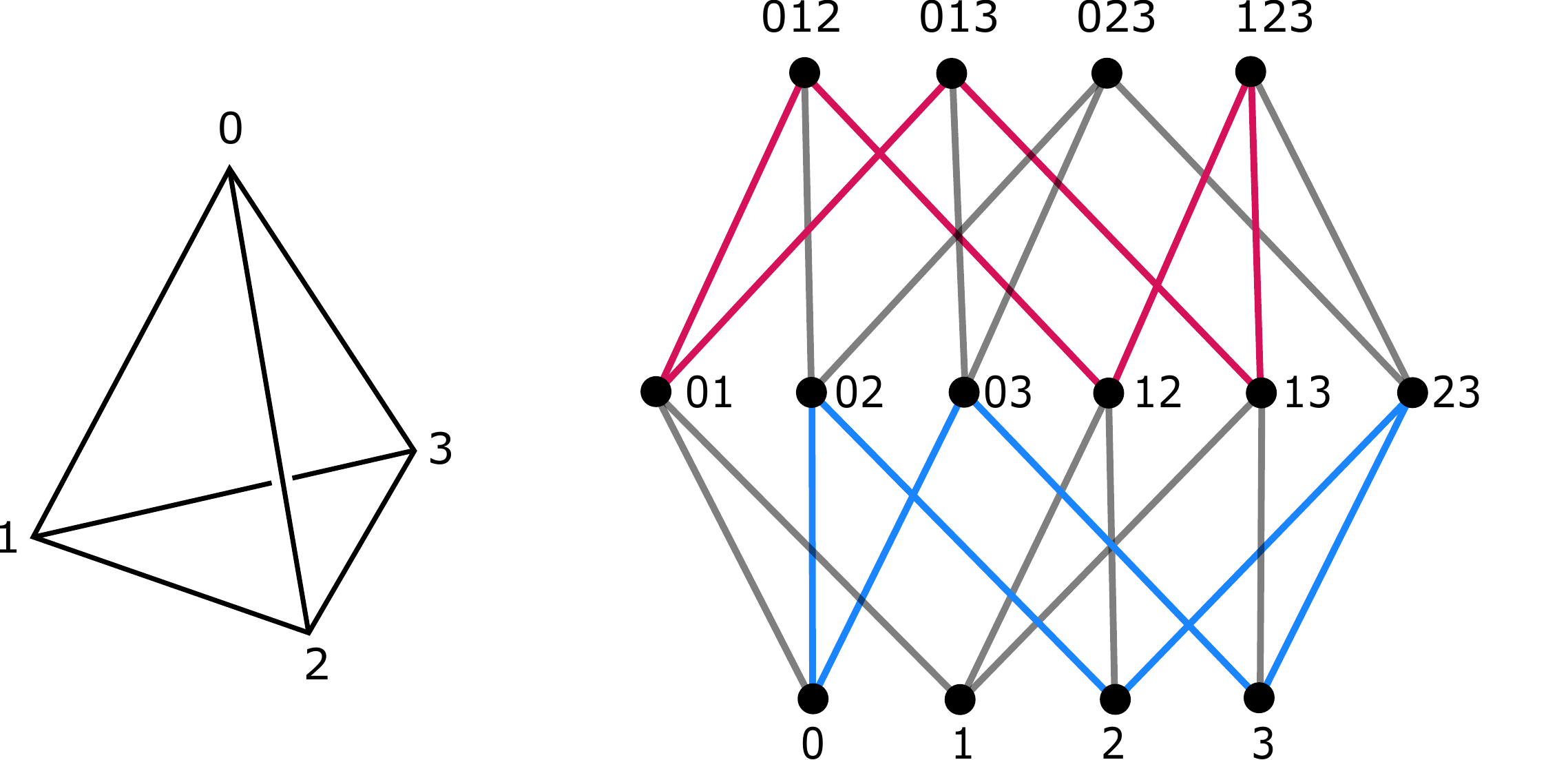}
\caption{The face posets for the $3$-simplex (with the vertex associated to the whole simplex removed);  we included two disjoint loops in blue and pink; these are supported by (at least) four matchings, corresponding to the possible ways of orienting them.}
\label{fig:hassesimplex}
\end{figure}
It is immediate to determine that $\eta (\Delta^2) = 1$ and $\eta (\Delta^3) = 2$ (note that we can already find the two cycles in Figure \ref{fig:hassesimplex} and we cannot create more because we would need six vertices disjoint from those already used); more generally it is possible to give the following coarse double bound $n-1 \le \eta(\Delta^n) \le \lfloor \frac{2^{n+1}-1}{6} \rfloor $, where $2^{n+1}$ is the number of vertices in the $(n+1)$-cube. The right inequality follows from the fact that the shortest oriented loops in $\mathcal{F}(\Delta^n)$ have length $6$.

Surprisingly, an explicit computation of the homology groups $H_* (\mathcal{M}(\Delta^n))$ seems to be currently out of reach. We can, however, complement the analysis carried out by Chari and Joswig in \cite[Sec. 5]{chari2005complexes}, by providing the computation of the reduced homology $\widetilde{H}_*(\text{M}_k(\Delta^n))$ for $n \le 3$ in Table \ref{tbl:simplexhom}.

\bgroup
\def\arraystretch{2}
\begin{table}[ht]
\begin{center}
\begin{tabular}{| Sc || Sc | Sc | Sc |}
\hline   
$X$ & $\widetilde{H}_*(\text{M}_0(X))$ & $\widetilde{H}_*(\text{M}_1(X))$ & $\widetilde{H}_*(\text{M}_2(X))$ \\
\hline
\hline
$\Delta^2$ & $\Z_{(1)}^4$ & $\Z_{(1)}^2$ & -- \\
\hline
$\Delta^3$ & $\Z_{(4)}^{99}$ & $\Z_{(4)}^{39}$ & $\Z_{(4)}^{39}$ \\
\hline
\end{tabular}
\end{center}
\caption{The reduced filtered homology of $\Delta^2$ and $\Delta^3$. }
\label{tbl:simplexhom}
\end{table}
\egroup
\end{exa}

\begin{exa}[\textbf{Complete graphs}]\label{ex:completeg}
The matching complex $\overline{\text{M}}(K_n)$ of the complete graph $K_n$ is a rather important and well-studied object, due to its many connections with several other branches of mathematics (see \cite[Ch.~1]{jonsson2007simplicial}). Somewhat surprisingly,  $\overline{\text{M}}(K_n)$ does contain torsion already for  $n=7$ (see \cite{jonsson2010more}, \cite{shareshian2007torsion} for several other related results). A table detailing the homology of $\overline{\text{M}}(K_n)$ can be found in \cite[Table~1]{jonsson2010more}.

In Table \ref{tbl:completehom} we display the (reduced) homology of the subcomplexes of $\text{M}(K_n)$ induced by the filtration $J$. Since loops in $K_n$ have length at least $3$, it follows that $\eta(K_n) = \lfloor \frac{n}{3} \rfloor$.

\bgroup
\def\arraystretch{2}
\begin{table}[ht]
\begin{center}
\begin{tabular}{| Sc || Sc | Sc | Sc |}
\hline   
$X$ & $\widetilde{H}_*(\text{M}_0(X))$ & $\widetilde{H}_*(\text{M}_1(X))$ & $\widetilde{H}_*(\text{M}_2(X))$ \\
\hline
\hline
$K_3$ & $\Z_{(1)}^4$ & $\Z_{(1)}^2$ & -- \\
\hline
$K_4$ & $\Z_{(2)}^{27}$ & $\Z_{(2)}^5$ & -- \\
\hline
$K_5$ & $\Z_{(3)}^{256}$ & $\Z_{(3)}^5 \oplus \Z_{(4)}^{23}$ & -- \\
\hline
$K_6$ & $\Z_{(4)}^{3125}$ & $\Z_{(4)}^6 \oplus \Z_{(5)}^{927}$ & $\Z_{(4)}^6 \oplus \Z_{(5)}^{967}$ \\
\hline
$K_7$ & $\Z_{(5)}^{46656}$ & $\Z_{(5)}^7 \oplus \Z_{(6)}^{23287}$ & $\Z_{(5)}^7 \oplus \Z_{(6)}^{25107}$ \\
\hline
\end{tabular}
\end{center}
\caption{The reduced filtered homology of $K_n$ for $n\leq 7$.}
\label{tbl:completehom}
\end{table}
\egroup

We point out a few things; the homology  $\widetilde{H}_*(\text{M}_0(K_{n+1}))$ is by \cite[Thm.~3.1]{kozlov1999complexes} isomorphic to $\Z_{(n-1)}^{n^n}$, which is consistent with these computations. Also, the homology of the complexes associated with $K_3$ can be computed using the simplicial complex in the central part of Figure~\ref{fig:circles}.
\end{exa}

\begin{defi}
If an oriented loop $\gamma$ in $\mathcal{F}(X)$, not sharing edges with any other loops, is supported by a matching $m$, then there is another matching $m'$, coinciding everywhere  with $m$ except on $\gamma$, where it induces the opposite orientation. The two matchings $m$ and $m'$ are said to differ by a \emph{click move}; any two matching that can be related by a finite sequence of click moves will be called \emph{click equivalent}. 
\end{defi}
As an example, the blue and pink matchings on the right of Figure \ref{fig:circles} differ by a click move.
~\\

We conclude the section with the following proposition asserting that for a graph $G$, a matched edge in $\mathcal{F}(G)$ can never belong to more than one oriented cycle supported by a matching (see also Figure \ref{fig:differenceingraphs}). This is in contrast with the face posets for any complex other than a graph, where this phenomena can occur (as in Figure \ref{fig:nondisjointloops}). As a consequence, $\eta (G)$ coincides with the maximal number of vertex-disjoint cycles in $G$, or  in other words 
\begin{equation}\label{eqn:etaforgraphs}
\eta(G) = \max_{C \in vd\mathcal{C}(G)} \{|C|\}.
\end{equation}
This quantity is well known and has been extensively studied (see e.g.~\cite{dirac1963maximal}, \cite{egawa2003covering} and references therein).

\begin{prop}\label{prop:vertexdisjoint}
If $G$ is a graph, then there is a surjective function from the set of \emph{oriented} loops in $\mathcal{F}(G)$ supported some by some matching, to  collections of \emph{unoriented} vertex-disjoint loops in $G$.
\end{prop}
\begin{proof}
Let $\overrightarrow{C}$ be a set of oriented loops in $\mathcal{F}(G)$, supported by a matching $m$. Let $\varphi(\overrightarrow{C})$ denote the subgraph of $G$ obtained by considering all edges in $G$ whose corresponding vertices in $\mathcal{F}(G)$ are incident to edges in the matching $m$. Then $\varphi(\overrightarrow{C})$ is composed of cycles in $G$, which are vertex-disjoint since $\overrightarrow{C}$ was supported by a matching (see Figure~\ref{fig:differenceingraphs}).

The map $\varphi$ is surjective: given any collection $C$ of vertex-disjoint loops in $G$, we can randomly orient each loop, either clockwise or counter-clockwise; to these oriented loops we can associate the matching on $\mathcal{F}(G)$ given by considering the half-edges of the corresponding loops in $B(G)$ that start from the initial vertex (with respect to the chosen orientation) and end at its barycenter. The image of a matching under $\varphi$ constructed this way is $C$.
\end{proof}
\begin{figure}[ht]
\centering
\includegraphics[width=9cm]{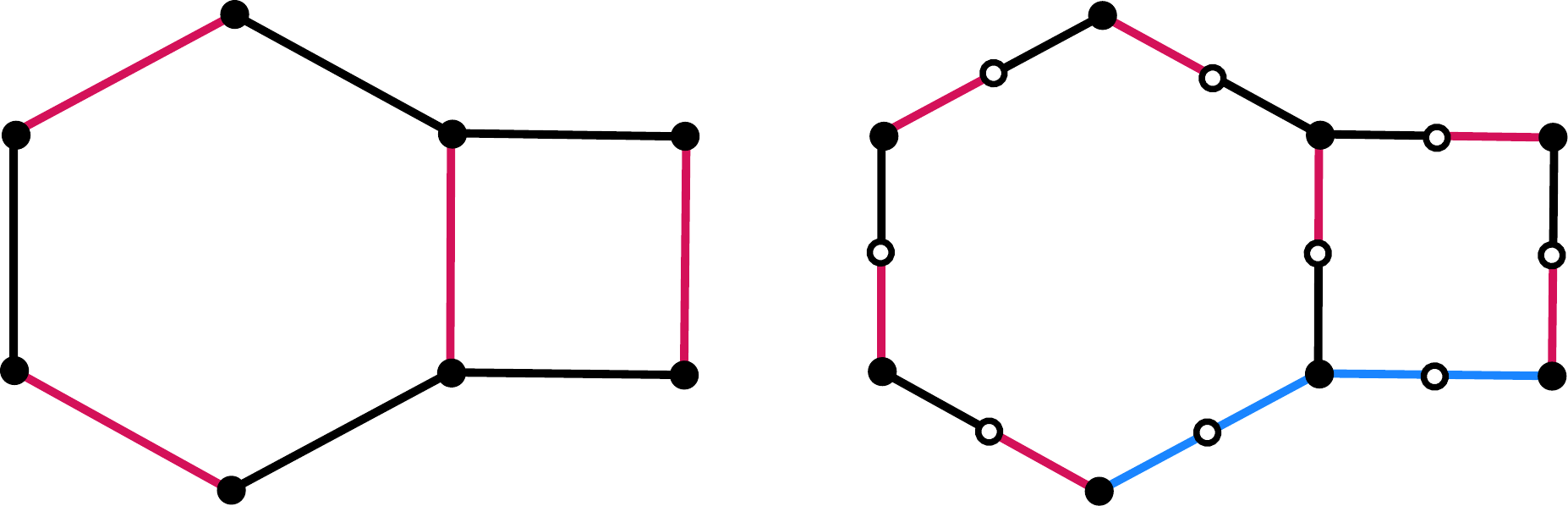}
\caption{A schematic explanation of why a matched edge in $\mathcal{F}(G)$ cannot be part of two different cycles, even if it is in $G$. There is no way of ``transporting'' the pink matching from left to right, as e.g.~it could not be extended on the two blue edges in lower-right part of the figure.
}
\label{fig:differenceingraphs}
\end{figure}

Note that any two matchings sharing the same image under $\varphi$ and coinciding outside of the oriented loops, are click equivalent.

\section{Filtered homology}\label{sec:filteredhomology}

We can use the function $J$ to define a filtration on the simplicial chain complex of $X$.
We will use $\F = \F_2$ coefficients throughout for simplicity, but the same construction also works with integer coefficients.
Define $$\widehat{C}_i^j(X) = \F \langle \sigma \in \text{M}(X) \,|\, dim(\sigma) = i, J(m_\sigma) = j \rangle,$$
and $$\widehat{C}(X) = \bigoplus_{i,j \in \N} \widehat{C}_i^j(X).$$
~\\
If $x \in \widehat{C}_i^j(X)$, we will refer to $i$ and $j$ as its \emph{homological} and \emph{filtration} degree respectively.

The simplicial differential $\partial$, defined in the usual way, can be split into a $J$-preserving component $\partial_J$ and a ``diagonal'' one $\partial_d$. The former one decreases the bi-degree $(i,j)$ by $(1,0)$, while the latter decreases the bi-degree by $(1,c)$ for some $c\ge 1$; this difference in behaviour stems from cases as in the top part of Figure \ref{fig:nondisjointloops}, where removing a single edge from a matching belonging to some cycle might decrease the number of oriented cycles by more than one. It follows from Proposition~\ref{prop:vertexdisjoint} that, in the case when $X$ is a graph, the bi-degree of $\partial_d$ as a filtered map can only be $(1,1)$.

\begin{lem}
The map $\partial_J$ is a differential on $\widehat{C}(X)$.
\end{lem}
\begin{proof}
It is immediate to see that for any simplex $\sigma \in \widehat{C}(X)$, $(\partial_J)^2 (\sigma)$ is a sum over all pairs of edges that can be removed from the matching $m_\sigma$ on $\mathcal{F}(X)$, while preserving the number of cycles supported by $m_\sigma$. The sum is $0$ with $\text{mod} \,2$ coefficients, since if there are at least two edges not contained in the oriented cycles supported by $m_\sigma$  there is an even number of ways of choosing such pairs in an ordered fashion. On the other hand, if there is only one such edge, then the horizontal differential of $m_\sigma$ is given by removing it, and applying $\partial_J$ again sends the collection of matched edges on the loops to $0$.
\end{proof}
\begin{rmk}
It is generally not true that $\partial_d$ is a differential; this can be appreciated by considering a matching akin to the one from Figure~\ref{fig:nondisjointloops}. Consider a matching $m$ supporting two loops $C = \{\gamma_1, \gamma_2 \}$, such that $\mu (C) = 1$ (so, the two loops share at least one edge). Assume for simplicity that $\gamma_1 \cap \gamma_2$ consists of a single edge $e$. To have $\partial_d^2 (m)=0$, we rely on the fact that removing two edges from $m$ is independent of the order in which they are removed. However, this is not the case if we consider the face of (the simplex corresponding to) $m$ obtained by removing two distinct edges $\{e_1,e\}$. Removing $e$ deletes both cycles $\gamma_1$ and $\gamma_2$, so $\partial_d$ on $m\setminus\{e\}$ is $0$. However, removing $e_1$ first, then $e$ is not $0$ since after removing $e_1$ we still have one cycle, $\gamma_2$.
\end{rmk}

\begin{rmk}~\label{rmk:diagonal}
Proposition~\ref{prop:vertexdisjoint} implies that a matching on the face poset of a graph $G$ can only support disjoint oriented loops, so the aforementioned phenomenon cannot occur in the matching complex of a graph $G$. In particular $\partial_d$ is a differential on $\widehat{C}(G)$.
\end{rmk}

Now we can define several homology groups associated with this decomposition of the simplicial differential:
\begin{defi}
Denote by $\widehat{H}^k(X)$ the homology of the subcomplex of $\widehat{C}(X)$ spanned by all simplices with filtration degree $\le k$, so that $$H_*(\text{M}_k(X)) = \widehat{H}_*^k(X).$$
We can consider two other related homologies  associated to the decomposition $$\partial = \partial_J + \partial_d.$$  Define the \emph{horizontal homology} $hH(X)$ as the bi-graded homology of the complex $(\widehat{C}(X), \partial_J)$. More explicitly, $hH(X)$ is the associated graded object to the filtration $J$ on the complex computing the simplicial homology of $\text{M}(X)$.

If $G$ is a graph, we can also define the \emph{diagonal homology} $dH(G)$ as the homology of $(\widehat{C}(G), \partial_d)$ (\emph{cf}.~Remark~\ref{rmk:diagonal}).
\end{defi}

\begin{figure}[ht]
\centering
\includegraphics[width=12cm]{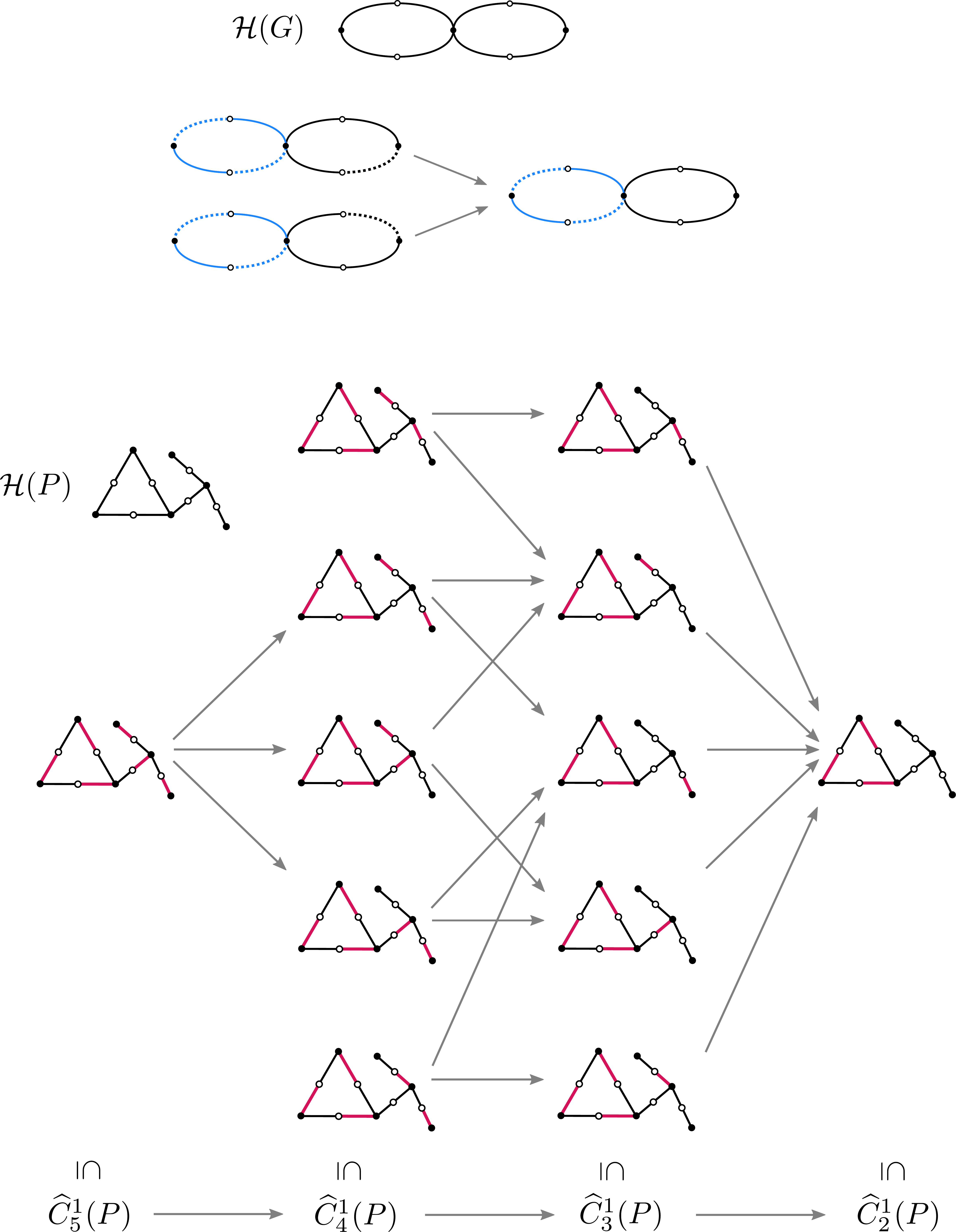}
\caption{The (unoriented) face poset for a graph $G$, which is homotopic to the bouquet of two circles, is shown at the top of the figure. Just below it, a portion of the complex $\left( \widehat{C}_*^1(G), \partial_J \right)$ is displayed, whose homology is $\F$ in homological degree $2$. Matched edges are dotted, oriented cycles are blue, and gray arrows represent the differential. In the lower part of the figure is part of the complex $\left( \widehat{C}_*^1(P), \partial_J \right)$, for a pseudo-tree $P$. All of the generators here share the same oriented cycle on the triangle. The homology of this subcomplex is trivial in all homological degrees. Choosing matchings inducing the opposite orientation on the triangle gives a separate complex with the same homology, as showed in Theorem \ref{thm:hH}.}
\label{fig:examplehomology}
\end{figure}

These homology groups potentially contain a great deal of information about the matching complex of $X$ and the intermediate complexes $\text{M}_k(X)$. Theorem~\ref{thm:hH} provides the computation of the horizontal homology $hH(X)$ in terms of simpler objects. 

As an example, if $T$ is a tree, then $hH(T) \cong  H(\mathcal{M}(T))$,
and
$dH(T) \cong \widehat{C}(\mathcal{M}(T)),$
because all simplicial chains are concentrated in filtration degree $0$, and in the latter case the diagonal differential is trivial. Before we can state our next result, we need to introduce a few terms. 

For an element $C \in \mathcal{C}(X)$, define its complement $X_C$ as the subgraph of $\mathcal{F}(X)$ given by all edges not incident to any vertices composing $C$. Note that the complement might not be the face poset of a simplicial subcomplex of $X$.

\begin{thm}\label{thm:hH}
Let $X$ be a finite simplicial complex; then
\begin{equation}\label{eqn:hH}
hH(X) \cong H(\mathcal{M}(X)) \oplus \left( \bigoplus_{C \in \mathcal{C}_\circ(X)}   \bigoplus^{2^{\mu_1 (C)}} \widetilde{H} \left( \overline{\mathcal{M}}(X_C)  \right)  \right) \oplus \left( \bigoplus_{C \in \mathcal{C}_M(X)}  \F^{2^{\mu_1(C)}}_{\langle|C|\rangle}  \right) 
\end{equation}
Here, $\widetilde{H}$ denotes reduced homology, and $\F^{(p)}_{\langle q\rangle}$ denotes $p$ generators of the homology in filtration degree $q$.
\end{thm}
\begin{proof}

We prove this theorem by considering the decomposition of oriented loops $\mathcal{C}(X) = \{\emptyset\} \sqcup  \mathcal{C}_\circ(X) \sqcup \mathcal{C}_M(X)$. If $C\in\mathcal{C}(X)$ is the empty set, then we are computing the homology of the chain complex of the discrete Morse complex, because this is precisely the simplicial complex generated by matchings supporting no oriented loops in $\mathcal{F}(X)$. This gives us the first term, $H(\mathcal{M}(X))$, in Equation \ref{eqn:hH}.

Before we can consider the contribution to $hH(X)$ from nonempty $C\in\mathcal{C}(X)$, observe that for such a $C$, the differential $\partial_J$ preserves the splitting of $\widehat{C}(X)$ into elements that support $C$ and those that do not. Moreover, the differential $\partial_J (m)$ of a matching $m$ supporting the  loops $\{\gamma_r\}$ that compose $C$ is either empty (if all edges comprising $m$ are on the loops) or a sum of other matchings sharing precisely the same edges on the cycles as $m$. 

As discussed in Example \ref{ex:circles}, if a cycle $\gamma$ does not share edges with any other one, there are exactly two different ways that a matching can support $\gamma$ (corresponding to the two possible orientations on $\gamma$); hence for each $C \in \mathcal{C}(X)$ the subcomplex given by the matchings that support $C$ splits further into $2^{\mu_1(C)}$ subcomplexes.  This follows from the fact that if two loops supported by a matching share at least one edge, then there is only one possible orientation that can be induced on them by any matching supporting both.

Moreover, if $C \in \mathcal{C}_M(X)$, then by definition all edges belonging to $m$ are contained in some loop in $C$, and thus their $J$-boundary is trivial. Since such a maximal set of cycles is spanning, it follows that all matchings supporting $C$ cannot be boundaries, and we obtain the right summand in Equation \eqref{eqn:hH}.

To conclude, we need to determine the homology of the subcomplexes given by matchings supporting non-maximal collection of cycles. Given a matching $m$, consider the subcomplex $\widehat{C}(\{\gamma_r\}) \subseteq \widehat{C}(X)$ spanned by all matchings that induce the same non-empty set of oriented loops $\{\gamma_r\}$ (which implies that they share the same edges as $m$ on these loops). As discussed above, the restriction $\restr{\partial_J}{{\widehat{C}(\{\gamma_r\})}}$ is still a differential; moreover the map $$f: C_*(\overline{\mathcal{M}}(X_C)) \longrightarrow \widehat{C}(\{\gamma_r\})$$ sending a matching $m^\prime \in C_*(\overline{\mathcal{M}}(X_C))$ to the matching $f(m') \in \widehat{C}(\{\gamma_r\})$ obtained by completing $m^\prime$ on the loops as $m$ is clearly a bijective chain map. The empty matching on $X_C$ gets mapped by $f$ to a generator of the complex, which is the unique $J$-boundary of all matchings of degree $1$ in $\widehat{C}(\{\gamma_r\})$ -- that is, all matchings that have exactly one edge of $\mathcal{F}(X)$ not on the oriented loops $\{\gamma_r \}$. This extra element has the effect of adding a ``terminal generator'' to the complex, and hence we get the reduced homology of $\overline{\mathcal{M}}(X_C)$ in the central summand of Equation~\eqref{eqn:hH}, rather than the usual homology (see the top of Figure~\ref{fig:examplehomology} for an example).
\end{proof} 

In the special case where the simplicial complex is a graph $G$, we can relate the decomposition of Theorem \ref{thm:hH} to the set $vd\mathcal{C}(G)$, using the correspondence provided by Proposition \ref{prop:vertexdisjoint}. In analogy with the previous result, let us denote $vd_\circ(G)$ the collection of non-maximal vertex disjoint cycles in $G$.

\begin{cor}\label{cor:hHgraphs}
Let $G$ be a graph, and denote by $G_C = B(G) \setminus B(C)$; then  
\begin{equation}\label{eqn:hHgraphs}
hH(G) \cong H(\mathcal{M}(G)) \oplus \left( \bigoplus_{C \in vd_\circ(G)}   \bigoplus^{2^{|C|}} \widetilde{H} \left( \overline{\mathcal{M}}(G_C)  \right)  \right) \oplus \left( \bigoplus_{C\;2\text{-factor of } G}  \F^{2^{|C|}}_{\langle|C|\rangle}  \right) 
\end{equation}
\end{cor}
\begin{proof}
This is just a restatement of Theorem \ref{thm:hH}, after noting that $X_C$ coincides with $G_C$, and that by Proposition~\ref{prop:vertexdisjoint}, the collection of maximal oriented cycles in $\mathcal{F}(G)$ corresponds to $2$-factors in $G$.
\end{proof}

In particular, the rank of $hH(G)$ is always bounded below by the sum $$\sum_{C\;2\text{-factor of } G} 2^{|C|}.$$

We collect here some simple computations of $hH$. 
\begin{exa}\label{ex:computehH}

If $X = C_n$, then we have seen already that $$hH_*^0(C_n) = H_*(\mathcal{M}(C_n)) = H_*(\overline{\mathcal{M}} (C_{2n}))$$ in Example~\ref{ex:circles}. The only other non-trivial group is $\Z^2$ in bi-degree $(n-1,1)$, generated by the two maximal simplices.

In a slightly more general case, if $P$ is a pseudo-tree, then we can write $P = C_P \cup_{i=1}^m T_i$, where $C_P$ is a cycle, and the $T_i$ are disjoint trees, each sharing a unique vertex $v_i$ with $C_P$.
Then, by Corollary~\ref{cor:hHgraphs}, if we let $\widetilde{T}_i = B(T_i) \setminus v_i,$ $$hH_*^1 (P) \cong 
\bigoplus^2 \widetilde{H}_* \left( \overline{\mathcal{M}}(\cup_{i=1}^m T_i) \right) \cong 
\bigoplus^2 \widetilde{H}_* \left( \overline{\text{M}}(\cup_{i=1}^m T_i) \right).$$
Note that by \cite[Thm.~4.13]{marietti2008uniform}, the latter  is isomorphic to $\Z$ (if the homotopy type of the  complex is that of a pair of points), or to the reduced homology of two disjoint wedges of spheres.
\end{exa}

The dimension of the simplices in complexes with non-zero  filtration can be bounded from below; in what follows, $\ell(\gamma)$ will denote the length of a loop $\gamma$.
\begin{lem}
If $G$ is a graph, then the subcomplex of $\widehat{C}(G)$ spanned by simplices supporting a fixed $C \in vd\mathcal{C}(G)$ is trivial in all dimensions strictly less than $ -1 + \sum_{\gamma \in C} \ell (\gamma)$. 

In particular, if $G$ is a graph without self-loops, the simplices in $C_*^j(G)$ have dimension $\ge 2j-1$; if furthermore $G$ is simple, than $C_*^j(G)$ is generated by simplices of dimension $\ge 3j-1$.
\end{lem}
\begin{proof}
By Proposition~\ref{prop:vertexdisjoint}, a matching belonging to $\widehat{C}^j_*$ supports exactly $j$ disjoint cycles $\{\gamma_i\}_{i = 1,\ldots ,j}$ on $G$. The number of edges matched on $\mathcal{F}(G)$ for each cycle $\gamma_i$ is given by $\ell(\gamma_i)$, therefore the minimal number of edges in a such a matching is given by $\sum_{i=1}^j \ell(\gamma_i)$. 

The second part of the statement follows by noting that, in a graph without self-loops, all cycles have length $\ge 2$, and if the graph is also simple, than the bound can be improved to $\ge 3$ because multiple edges between pairs of vertices are not allowed.
\end{proof}

~\\
It is possible to consider the \emph{decategorification} of $hH(X)$ (\emph{cf.} \cite{khovanov2000categorification}), or in other words the polynomial obtained as the graded Euler characteristic of $hH(X)$, defined as  $$\chi_t (hH(X)) = \sum_{i,j \in \N} (-1)^i rk(hH_i^j(X)) \cdot t^j.$$

\sloppy
That is, the coefficient of $t^m$ is the Euler characteristic of the chain complex $(\oplus_{i \in \N} \widehat{C}_i^m(X), \partial_J)$. In particular, the constant term is $\chi(\mathcal{M}(X))$. 

If, $G$ is a graph, denote by $\rho_k(G)$ the number of rooted spanning forests in $G$ with exactly $k$ edges. Then, using \cite[Thm.~7.5]{biggs1993algebraic} combined with Proposition~\ref{prop:chari} we get the equality:  
\begin{equation}
\restr{\chi_t(hH(G))}{t=0} = \sum_{k = 0}^{|E(G)|}(-1)^k \cdot \rho_k(G) = p_G(1),
\end{equation}
\fussy
where $p_G(x) = det(L-x\cdot \text{Id})$ is the characteristic polynomial of the Laplacian matrix of $G$. In fact, if we write 
\begin{equation}
p_G(x) = \sum_{i = 0}^{|E(G)|} c_{i} \cdot x^{|E(G)|-i},
\end{equation}
then $(-1)^k c_k = \rho_k(G)$.

More generally, whenever $j>0$, by Corollary~\ref{cor:hHgraphs} the coefficient of $t^j$ in $\chi_t(hH(G))$ for a graph $G$ is given by  $$\sum_{\substack{C \in vd_\circ \mathcal{C}(G)\\ |C| = j}} 2^j \cdot \left( \chi(\overline{\mathcal{M}}(G_C)) + (-1)^{\ell(C) -1} \right) + \sum_{\substack{C\;2\text{-factor of } G\\ |C| = j}} (-1)^{\ell(C) -1} \cdot 2^j.$$
Here, the extra $(-1)^{\ell(C) -1}$ is due to the presence of the terminal elements described in the proof of Theorem~\ref{thm:hH}, and $\ell(C)$ denotes the sum of the lengths of the cycles in $C$. We point out that Equation~\eqref{eqn:etaforgraphs} implies that the maximal degree of $\chi_t(hH(G))$ coincides with $\eta(G)$.

\begin{exa}\label{ex:decat}
We  can easily compute the decategorification of $hH$ for the cycle graphs $C_n$:  there are only two simplices in $J$-degree $1$, and 
\begin{equation}\label{eq:casesdeterminant}
p_{C_n}(1) = \begin{cases} \:0 & \text{ if } n \equiv 0\:\;\;\:\:\:\: \text{mod}\; 6\\
-1 & \text{ if } n \equiv 1,5\;\:\: \text{mod}\; 6\\
-3 & \text{ if } n \equiv 2,4\;\:\: \text{mod}\; 6\\
-4 & \text{ if } n \equiv 3\;\;\:\:\:\:\:\text{mod}\; 6
\end{cases}
\end{equation}
This can be deduced by computing the characteristic polynomial of the Laplacian of $C_n$, together with some elementary linear algebra. More precisely, $p_{C_n}(1)$ is the determinant of the matrix $L_n-\text{Id}$, where 
$$L_n = \begin{pmatrix}
2  & -1 & 0 & \ldots & 0 & -1\\
-1 & 2 & -1 & \ldots & 0 & 0\\
0 & -1  & 2 &-1 & \ldots  & 0\\
\vdots & 0 &  \ddots & \ddots & 2 & -1 \\
-1 & 0 & \ldots & 0 &-1 & 2 
\end{pmatrix}.$$
The result in Equation~\eqref{eq:casesdeterminant} then follows by considering the Laplace expansion of the determinant, combined with the fact that the determinant $D(n)$ of the tri-diagonal matrix $\text{Trid}_n(-1,1,-1)$ satisfies the recursive relation $$D(n) = D(n-1) - D(n-2),$$ with initial conditions $D(1) = 1$ and $D(2) = 0$. The sequence $D(n)$ is easily seen to be periodic of period $6$.

Therefore, $$\chi_t(C_n) = (-1)^{n-1} 2t + p_{C_n}(1).$$
\end{exa}

~\\
The diagonal homology appears to be more complex; however, it is possible to give a graph-theoretic interpretation of its ranks for some specific class of graphs. 
\begin{prop}
The top-dimensional diagonal homology group of a pseudo-tree $P$ is free of rank $\rho_{|E(P)|-1}(P) -2$. 
Furthermore, if $P = C_n$ the homology $dH_*^*(C_n)$ is trivial in filtration degree $1$, and the  ranks in bidegree $(i,0)$ are $\rho_{i+1}(C_n)$ if $i < n-2$, and $\rho_{n-1} (C_n) -2$ if $i = n-2$;  these are the only non-trivial groups.
\end{prop}
\begin{proof}
The first barycentric subdivision of a pseudo-tree has exactly two perfect matchings, which are related by a click move on the cycle of $P$ (as in Figure~\ref{fig:circles}). Therefore, the complex has only two elements in top bi-degree $(|E(P)| -1, 1)$. The diagonal differential maps them to distinct elements (those obtained by removing a single edge from the cycle), hence it is injective. 
The diagonal differential from $\hat{C}_{|E(P)| -2}^0$ is $0$ since there are no loops to remove (the same is true for all other complexes in filtration degree $0$). Since the diagonal differential from $\hat{C}_{|E(P)| -1}^1 $ is injective, the first result follows.

By combining \cite[Thm.~7.5]{biggs1993algebraic} and Proposition~\ref{prop:chari}, there are exactly $\rho_{|E(P)|-1}(P)$ elements in bi-degree $(|E(P)|-2,0)$, hence the first statement follows.

For the second part, let us note that $\widehat{C}_*^1(C_n)$ is only non-trivial in homological degree $n-1$ (\emph{cf}.~Example~\ref{ex:computehH}), where it contains two basis elements. We can then conclude as in the previous case, by recalling that there are exactly $\rho_{i+1}(C_n)$ elements in bi-degree $(i,0)$ for $i = 0, \ldots, n-2$.
\end{proof}

\begin{figure}[ht]
\centering
\includegraphics[width=5.5cm]{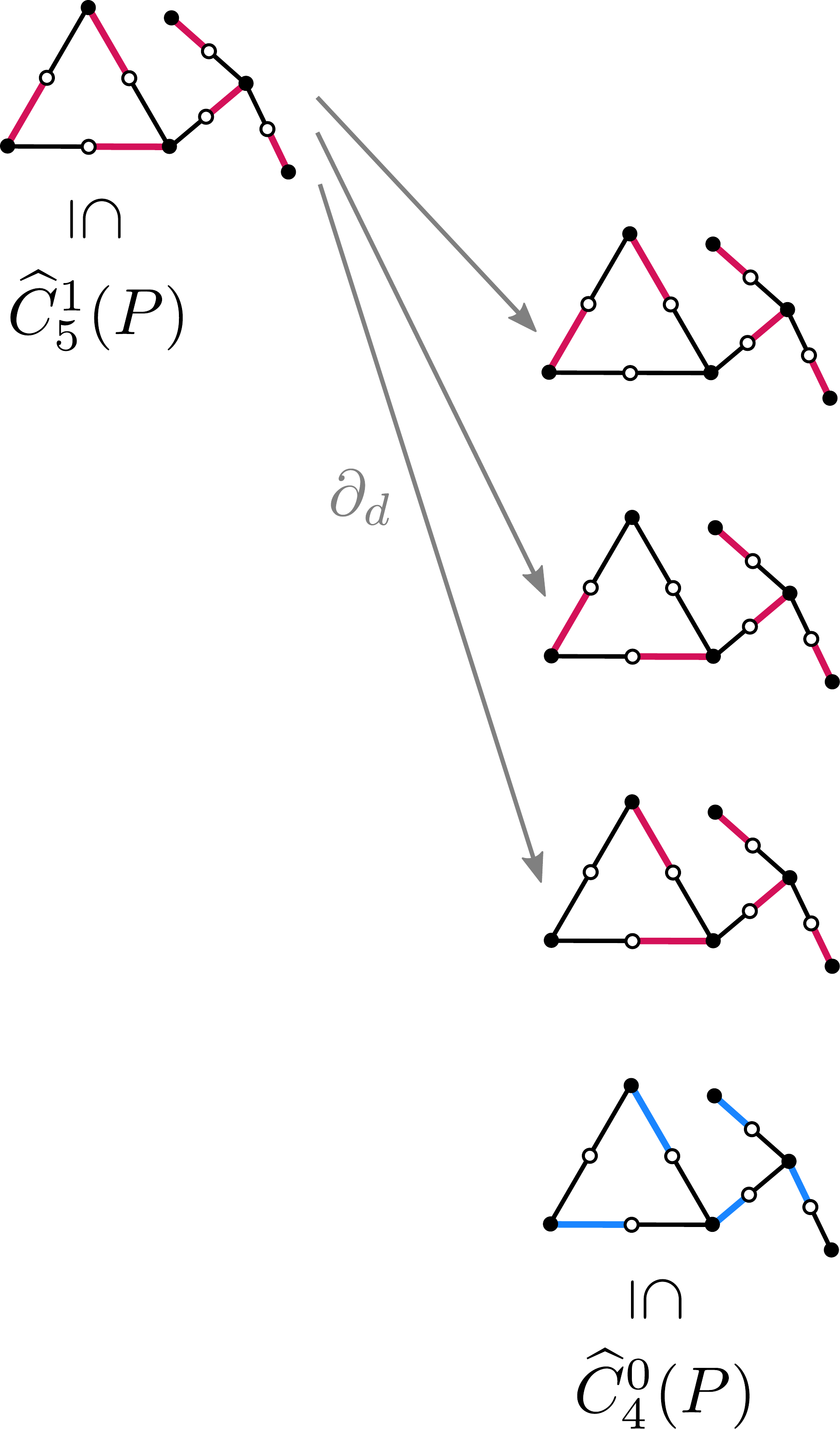}
\caption{A portion of the diagonal complex $\left(\widehat{C}_*^* (P), \partial_d \right)$ for the pseudo-tree from Figure~\ref{fig:examplehomology}. In the top-left we display one of the two chains in maximal bi-degree (the other one is obtained by clicking the length $3$ cycle). The gray arrows represent the diagonal differential; we also included in blue one (of the many) elements in $\widehat{C}_4^0(P)$ not belonging to the image of $\partial_d$.}
\label{fig:diagonalpseudotree}
\end{figure}

\subsection*{Acknowledgements} DC is supported by the European Research Council (ERC) under the European Unions Horizon 2020 research and innovation programme
(grant agreement No 674978). NY is supported by The Alan Turing Institute under the EPSRC grant EP/N510129/1. The authors would like to thank Luciana Basualdo Bonatto, Vidit Nanda and Agnese Barbensi for their useful feedback, and the Oxford Mathematical Institute for providing access to computational resources. The authors also want to thank the anonymous referees for their useful comments.

\end{document}